\documentclass{amsart}

\usepackage{amssymb}
\usepackage{amsmath}
\usepackage{amsthm}

\usepackage{graphicx}

\usepackage{amscd}
\usepackage[cmtip,all]{xy}

\newtheorem{theorem}{Theorem}[section]
\newtheorem{lemma}[theorem]{Lemma}

\theoremstyle{remark}

\begin{document}

\title[Rings $K(s)^*(BG)$]{Morava $K$-theory rings of the extensions of $C_2$ by the products of cyclic 2-groups.}
\author{Malkhaz Bakuradze and Natia Gachechiladze}
\address{Iv. Javakhishvili Tbilisi State University, Faculty of Exact and Natural Sciences}
\email{malkhaz.bakuradze@tsu.ge,  natia.gachechiladze@tsu.ge}
\thanks{First author was supported by Volkswagen Foundation, Ref. 1/84 328 and Rustaveli Foundation grant DI/16/5-103/12.}

\begin{abstract}
In \cite{SCH1} Schuster proved that $mod$ 2 Morava $K$-theory $K(s)^*(BG)$ is evenly generated for all groups $G$ of order 32.
There exist 51 non-isomorphic groups of order 32. In \cite{H}, these groups are numbered by $1, \cdots ,51$.
For the groups $G_{38},\cdots, G_{41}$, that fit in the title,
the explicit ring structure is determined in \cite{BJ}.
In particular, $K(s)^*(BG)$ is the quotient of a polynomial ring in 6 variables over $K(s)^*(pt)$ by an ideal generated by explicit polynomials. In this article we present some calculations using the same arguments in combination with a theorem of \cite{B0} on good groups in the sense of Hopkins-Kuhn-Ravenel. In particular, we consider the groups $G_{36},G_{37}$, each isomorphic to a semidirect product $(C_4\times C_2\times C_2)\rtimes C_2$, the group $G_{34}\cong (C_4\times C_4)\rtimes C_2$ and its non-split version $G_{35}$. For these groups the action of $C_2$ is diagonal, i.e., simpler than for the groups $G_{38},\cdots, G_{41}$, however the rings $K(s)^*(BG)$ have the same complexity.
\end{abstract}

\maketitle{}

{\centering\section{Introduction}}

In this paper we apply the transfer methods to the study of Morava K-theory rings of finite groups. 

%
%
%
%
%

Among many obvious motivations for studying  Morava K-theory rings of finite groups let us mention that elliptic cohomology rings of classifying spaces of finite groups provide especially suitable “testbed” for studying possible features of tentative elliptic objects, as demonstrated by works of Baker, Thomas, Devoto and others.

For various examples of finite groups the Morava K-theory ring has been shown in the works of N. Yagita to be generated by Chern classes. M. Hopkins, N. Kuhn and D. Ravenel introduced the notion of good groups in the sense that Morava K-theory of $BG$ is generated by transferred Euler classes of complex representations of subgroups of $G$.  This clearly indicates that the part of the relations which can be derived from the properties of the transfer should play decisive role in determining the whole ring structure. Note however not all groups are good by \cite{K}.

In \cite{TY1}, \cite{Y2}, \cite{Y3}, \cite{ST} the multiplicative structure has only been determined modulo certain indeterminacy. In   
\cite{SCH1}, \cite{SCH2}, \cite{SCH3}, \cite{SCH4}, \cite{SCHY} some artificial generators are suggested in order to obtain explicit relation.

 In the works by McClure and Snaith, Hunton and others cohomology groups of homotopy orbit spaces have been constructed, however it is especially interesting to express the ring structure purely in terms of transferred characteristic classes. Thus we are led to consider for finite coverings the interaction of transfers and Chern classes along the lines taken in \cite{BP}. Initial results for some examples given in \cite{BV}, \cite{B1}, \cite{B2} and \cite{BJ} are the first ones to describe multiplicative structure completely in terms of Euler classes and transferred Euler classes

\bigskip

\bigskip

\section{Preliminaries}

Let $K(s)^*(BG)$ be the $s$-th Morava $K$-theory of $BG$ at prime $p=2$. Note that by \cite{JW}, the coefficient ring $K(s)^*(pt)$ is the Laurent polynomial ring in one variable $\mathbb{F}_2[v_s,v_s^{-1}]$, where $\mathbb{F}_2$ is the field of 2 elements and $deg(v_s)=-2(2^s-1)$. So the coefficient ring is a graded field in the sense that all its graded modules are free, therefore Morava $K$-theories enjoy the Kunneth isomorphism. In particular, we have for the cyclic $2$-groups

$$K(s)^*(BC_{2^{n}}\times BC_{2^{m}})=\mathbb{F}_2[v_s,v_s^{-1}][u,v]/(u^{2^{ns}},v^{2^{ms}}),$$
where $x$ and $y$ are Euler classes of canonical complex linear representations.

\medskip

Recall from \cite{HKR} the following definition.

\medskip

a) For a finite group $G$, an element $x\in K(s)^*(BG)$ is good
if it is a transferred Euler class of a complex subrepresentation of $G$, i.e., a class
of the form $Tr^*(e(\rho))$, where $\rho$ is a complex representation of a subgroup $H<G$, $e(\rho)\in K(s)^*(BH)$ is its Euler class (i.e., its top Chern class, this being defined since $K(s)^*$ is a complex oriented theory), and $Tr:BG \to BH$ is the stable transfer map.

(b) $G$ is good if $K(s)^*(BG)$ is spanned by good elements as a $K(s)^*$–module.

\medskip

The good groups in the sense, that  $K(s)^{odd}=0$ (this fact is in principle weaker than being good in the sense of Hopkins-Kuhn-Ravenel) also play a role in the literature \cite{SCH1,SCH2,SCH3,Y3}.

In \cite{K}, Kriz proved a theorem about the Serre spectral sequence \cite{HKR}
\begin{equation*}
E_2=H^*(BC_p,K(s)^*(BH))\Rightarrow K(s)^*(BG)
\end{equation*}
associated to a group extension $1\to H \xrightarrow[]{}G \xrightarrow[]{} C_p\to 1$ for a $p$-group $G$ and normal subgroup $H\lhd G$ with $K(s)^{odd}(BH)=0$.
Namely in \cite{K} Kriz proved that $K(s)^{odd}(BG)=0$ if and only if the integral Morava
$K$-theory $\tilde{K}(s)^*(BH)$ is a permutation module for the action of $G/H \cong C_p$.
Kriz in \cite{K} and Yagita in \cite{Y3} proved that an extension of a cyclic $p$-group by elementary abelian
$p$-group satisfies the even-dimensionality conjecture.
For primes greater than 3, groups of $p$-rank $2$  were shown to have even Morava $K$-
theory by Tezuka-Yagita \cite{TY1} and Yagita \cite{Y2}. Furthermore, Yagita also
proved that these groups are generated by transfered Euler classes and are thus
good in the sense of Hopkins-Kuhn-Ravenel. Hunton \cite{HU} used "unitary-like embeddings" to show that if
$K(s)^*(BH)$ is concentrated in even degrees, then so is $K(s)^*(BH\wr C_p)$. An
independent proof of the same fact in the sense HKR was given in \cite{HKR};

One can consider the good groups in the stronger sense, i.e., $K(s)^*(BG)$ is generated by good elements as a $K(s)^*$-algebra
and all generating relations are consequences of formal properties of the characteristic classes and transfer.

There exist 51 non-isomorphic groups of order 32.  Among these groups,
the groups with numbers $3$, $14$, $16$, $31$, $34$, $39$ and $41$ can be presented as a semidirect product $(C_4\times C_4)\rtimes C_2 $, and the groups  $G=G_{36},G_{37}, G_{38}$- as $(C_4\times C_2 \times C_2)\rtimes C_2$.

We consider the examples $G_{34},\cdots, G_{37}$ to apply the following theorem of \cite{B0} on good groups in the sense HKR, which in our case  reads as

\begin{theorem}
\label{1}
Let $H_i$ and $G_i$ be finite $2$-group, $i=1, 2$, such that $H_i$ is good and  $G_i$ fits into an extension $1\to H_i\to G_i\to C_2 \to 1$.

Let $G$ fits into an extension of the form $1\to H \to G\to C_2 \to 1$, with diagonal conjugation action of $C_2$ on $H=H_1\times H_2 $.  Denote by

$Tr^*=Tr^*_{\varrho}: K(s)^*(BH)\to K(s)^*(BG)$, the transfer homomorphism, associated to $2$-covering $\varrho=\varrho(H,G):BH \to BG$,

$Tr_i^*=Tr^*_{\varrho_i}:K(s)^*(BH_i)\to K(s)^*(BG_i)$, the transfer homomorphism, associated to $2$-covering  $\varrho_i=\varrho(H_i,G_i):BH_i \to BG_i$, $i=1, 2$,

$\rho_i:BG\to BG_i $, the map, induced by projection $H \to H_i$ on $i$-th factor,

and let $\rho^*$ be the restriction of
$$(\rho_1,\rho_2)^*:K(s)^*(BG_1\times BG_2)\to K(s)^*(BG)$$
on $K(s)^*(BG_1)/Im Tr^*_1 \otimes K(s)^*(BG_2)/ImTr^*_2$.  Then

\medskip

i) If $G_i$ are good so is $G$.

ii) $K(s)^*(BG)$ is spanned, as a $K(s)^*(pt)$-module,
by elements of $ImTr^*$ and $Im\rho^*$

\end{theorem}

\bigskip

For our examples the action of $G/H\cong C_2$ is diagonal, i.e., is simpler than for the groups $G_{38},\cdots, G_{41}$, \cite{BJ}, but as we shall see the ring structures have the same complexity.

To prove that Theorems \ref{thm:G} and \ref{thm:G34-35} below, indeed give the ring structure we should

\medskip

(i) check the relations stated, and then to establish two facts:

(ii) the classes defined generate, and

(iii) the list of relations is complete.

\medskip

For the first step (i) our tool is the formal properties of the transfer homomorphism.

Let $H\lhd G$ be of index 2. Consider the double covering $\varrho : BH \to BG$. Let
$$Tr^*=Tr^*_{\varrho}=Tr^*(H,G): K(s)^*(BH)\to K(s)^*(BG)$$
be the associated transfer homomorphism induced by the stable transfer map \cite{A}, \cite{KP}, \cite{D}.

For all formal properties enjoyed by the transfer, Frobenius reciprocity and the
double coset formula, the best general reference is \cite{SCH4}.

\medskip

We will need the following transfer formula from \cite{BP}, which does not work for Morava $K$-theory at $s=1$. Therefore we restrict to $s>1$. Also, we set $v_s=1$ throughout the paper.

\medskip

Let $\xi \to BH$ be a complex line bundle and $\xi_{\varrho}=Ind_H^G(\xi)$ be its Atiyah transfer. Then
\begin{equation}
\label{eq:tr}
c_1(\xi_{\varrho})=
c_1(\psi)+\sum_{i=1}^{s-1}c_1(\psi)^{2^s-2^i}c_2(\xi_{\varrho})^{2^{i-1}}
+Tr^*(c_1(\xi)),
\end{equation}
where $\psi \rightarrow BG$ is the pullback of the canonical line bundle over $B\mathbb{Z}/2$ along the map $BG \rightarrow B\mathbb{Z}/2$ classifying $\varrho$.

\medskip

For the Chern classes $u$, $v$ and $Tr^*=Tr^*(H,G)$ the formula \eqref{eq:tr} implies

\begin{equation}
\label{Tru}
Tr^*(u)=c+x_1+\sum_{i=1}^{s-1}c^{2^s-2^i}x_2^{2^{i-1}}
\end{equation}
and

\begin{equation}
\label{Trv}
Tr^*(v)=c+y_1+\sum_{i=1}^{s-1}c^{2^s-2^i}y_2^{2^{i-1}}.
\end{equation}

\medskip

One has the following approximation formula for the formal group law in Morava $K$-theory (\cite{BV}, Lemma 2.2 ii)).

\begin{equation}
\label{eq:FGL}
F(x,y)=x+y+\Phi(x,y)^{2^{s-1}},
\end{equation}
where $\Phi(x,y)=xy+(xy)^{2^{s-1}}(x+y)\,\,\, modulo \,\, (xy)^{2^{s-1}}(x+y)^{2^{s-1}}.$

\bigskip

We also will need the following Lemma \cite{BJ}

\begin{lemma}
\label{lem:zeta^2}
The tensor square of a complex plane vector bundle $\zeta$ has the following total Chern class
$C(\zeta^{\otimes 2})=(1+c_1^2(det\zeta))(1+c_1^{2^s}(\zeta)+c_2^{2^s}(\zeta)).$
\end{lemma}

\bigskip

For the step (ii) our principal tool is the Serre spectral sequence (see \cite{BU1}, \cite{BU2}, \cite{HKR} for details)
\begin{equation}
\{E_2^{*,*}(BG)\}=H^*(BC_2,K(s)^*(BH))\Rightarrow K(s)^*(BG)
\end{equation}
associated to a group extension $1\to H \xrightarrow[]{i}G \xrightarrow[]{\pi} C_2\to 1$.

Here $H^*(BC_2,K(s)^*(BH))$
denotes the ordinary cohomology of $C_2$ with coefficients in the $\mathbb{F}_2[C_2]$-module $K(s)^*(BH)$, where the action of $C_2$
is induced by conjugation in $G$.

As a $C_2$-module $K(s)(BH)$ is a direct sum of free $F$ and trivial $T$ modules.

Here

\begin{equation*}
H^i(BC_2,F)=
\begin{cases}
[F]^{C_2}& \text{for $i=0$} \\
0        & \text{for $i>0$}.
\end{cases}
\end{equation*}

and
$$
H^*(BC_2,T)=H^*(BC_2)\otimes T.
$$

Recall $E_2^{0,*}$ is isomorphic to $[K(s)^*(BH)]^{C_2}$ via $\varrho^*: K(s)^*(BG) \to K(s)^*(BH)$. Also, via $\pi^*$, the spectral sequence $E_2^{*,*}(BG)$ is a module over the Atiyah-Hirzebruch spectral sequence $E_2^{*,*}(BC_2)$ that converges to $K(s)^*(BC_2)$.

The Frobenius reciprocity of the transfer says that the composition $\varrho^*Tr^*$ is the trace map 

\begin{equation}
\label{t}
\varrho^*Tr^*=1+t, \text { where $t$ is the generator $t\in G/H\cong C_p$.}
\end{equation}

Clearly, the trace map $\varrho^*Tr^*$ always maps the good elements of $K(s)^*(BH)$ onto $[F]^{C_2}$.
Hence the group $G$ is good iff $T$ is also covered by $\varrho^*$-images of good elements.

In general if to operate roughly, even  for $s=2$, this requires a serious computational effort, see \cite{SCH4} p.78. One can to simplify this task by suggesting the special bases for particular $C_2$-modules $K(s)^*(BH)$.  This simple but convenient idea first was realized in \cite{BJ}.

Finally one can test (iii) by calculating the rank $\chi_{s}(G)=\frac{1}{2}16^{s}-\frac{1}{2}4^{s}+8^{s}$
from the relations. The answer is calculated in  \cite{SCH1} using the Euler characteristic formula of Hopkins-Kuhn-Ravenel.

\bigskip

\bigskip

\section{Groups $G_{36},\,G_{37}$ }

The Morava ring for $G_{38}$, a semidirect product $\cong (C_4\times C_2\times C_2)\rtimes C_2$ was computed in \cite{BJ}. The relations for the groups $G_{36}$, $G_{37}$ can be treated similarly (see Section 5). The fact that transferred Chern classes generate (see Subsection 5.2) uses rather different arguments.

Let us recall the presentations

\begin{multline*}
\shoveleft{}\\
\shoveleft{G_{36}=\langle \mathbf{a},\mathbf{b},\mathbf{c} \mid  \mathbf{a}^4=\mathbf{b}^4=\mathbf{c}^2=[\mathbf{b},\mathbf{c}]=1, \mathbf{a}^{-1}\mathbf{b}\mathbf{a}=\mathbf{b}^{-1}, \mathbf{c}\mathbf{a}\mathbf{c}=\mathbf{a}^{-1} \rangle ,}\\
\shoveleft{G_{37}=\langle \mathbf{a},\mathbf{b},\mathbf{c} \mid  \mathbf{a}^4=\mathbf{c}^2=\mathbf{d}^2=[\mathbf{b},\mathbf{c}]=1,  \mathbf{d}=[\mathbf{a},\mathbf{c}], \mathbf{b}^2=\mathbf{a}^2, \mathbf{b}\mathbf{a}\mathbf{b}^{-1}=\mathbf{a}^{-1} \rangle .}\\
\end{multline*}
Let $H\cong C_4 \times C_2 \times C_2$ be the maximal abelian subgroup in $G$: $H=\langle \mathbf{b}, \mathbf{c}, \mathbf{a}^2\rangle$  for $G=G_{36}$; $H=\langle \mathbf{b}, \mathbf{c}, \mathbf{d}\rangle$ for $G=G_{37}$.
Let $\lambda$, $\mu$ and $\nu$  
be the complex line bundles over $BH$, the pullbacks of the canonical complex line bundles along the projections onto the first, second and third factor of $H$ respectively, that is,

for $H \lhd G_{36}$, 
$$\lambda(\mathbf{b})=i, \nu(\mathbf{a}^2)=\mu(\mathbf{c})=-1, \lambda(\mathbf{a}^2)=\lambda(\mathbf{c})=\nu(\mathbf{b})=\nu(\mathbf{c})=\mu(\mathbf{b})=\mu(\mathbf{a}^2)=1,$$
and for $H \lhd G_{37}$, 
$$\lambda(\mathbf{b})=i, \mu(\mathbf{c})=\nu(\mathbf{d})=-1, \lambda(\mathbf{c})=\lambda(\mathbf{d})=\mu(\mathbf{b})=\mu(\mathbf{d})=\nu(\mathbf{b})=\nu(\mathbf{c})=1.$$

The quotient of $G$ by the center $Z\cong C_2^2$ is isomorphic to $C_2^3$. The projections of $G$ on three factors of  $C_2^3$ induce line bundles 
$\alpha$, $\beta$ and $\gamma$ over $BG$:
$$\alpha(b)=\beta(c)=\gamma(a)=-1,\,\,\,\alpha(a)=\alpha(c)=\beta(a)=\beta(b)=\gamma(b)=\gamma(c)=1.
$$

\medskip

Let us denote Chern classes by

\begin{equation*}
x_i=c_i(Ind_H^G(\nu)); \,\,\,y_i=c_i(Ind_H^G(\lambda));
\end{equation*}

\begin{equation*}
a=
\begin{cases}
c_1(\alpha), \text{ for $G=G_{36}$ }, \\
c_1(\alpha\gamma), \text{ for $G=G_{37}$},
\end{cases}
\end{equation*}

and for both cases

\begin{equation*}
b=c_1(\beta),\,\,\,c=c_1(\gamma).
\end{equation*}

\medskip

Let $Tr^*: K(s)^*(BH)\to K(s)^*(BG)$ be the transfer homomorphism \cite{A} associated to the double covering $\rho: BH \to BG$ and let
$$
T=Tr^*(uv),
\text{ where } u=c_1(\nu);\,\,\, v=c_1(\lambda).
$$

\medskip

One our main result is the following

\bigskip

\medskip

\begin{theorem}
\label{thm:G}  Let $G$ be one of the groups $G_{36},G_{37}$. Then

i) $K(s)^*(BG)\cong K(s)^*[a,b,c,x_2,y_2,T]/R$, where the ideal $R$ is generated by

$a^{2^s}$, $b^{2^s}$, $c^{2^s}$,

$c(c+x_1+\sum_{i=1}^{s-1}c^{2^s-2^i}x_2^{2^{i-1}})$,
\,\,$c(c+y_1+\sum_{i=1}^{s-1}c^{2^s-2^i}y_2^{2^{i-1}})$,

$a(a+y_1+\sum_{i=1}^{s-1}a^{2^s-2^i}y_2^{2^{i-1}})$,\,\,\,
$b(b+x_1+\sum_{i=1}^{s-1}b^{2^s-2^i}x_2^{2^{i-1}})$,

$(c+y_1+\sum_{i=1}^{s-1}c^{2^s-2^i}y_2^{2^{i-1}})(b+x_1+\sum_{i=1}^{s-1}b^{2^s-2^i}x_2^{2^{i-1}})+b^{2^s-1}T$,

$(c+x_1+v_s\sum_{i=1}^{s-1}c^{2^s-2^i}x_2^{2^{i-1}})(a+y_1+\sum_{i=1}^{s-1}a^{2^s-2^i}y_2^{2^{i-1}})+a^{2^s-1}T$,

$T^2+Tx_1y_1+x_2y_1(c+y_1+\sum_{i=1}^{s-1}c^{2^s-2^i}y_2^{2^{i-1}})+x_1y_2(c+x_1+\sum_{i=1}^{s-1}c^{2^s-2^i}x_2^{2^{i-1}})$,

$T(b+x_1+\sum_{i=1}^{s-1}b^{2^s-2^i}x_2^{2^{i-1}})+b^{2^s-1}x_2(c+y_1)$,

$T(a+y_1+\sum_{i=1}^{s-1}a^{2^s-2^i}y_2^{2^{i-1}})+a^{2^s-1}y_2(c+x_1)$,\,\,\,$cT$, and

\begin{equation*}
x_2^{2^s}+
\begin{cases}
c^2+bc& \text{for $G=G_{36}$},\\
bc & \text {\,\,\,$G=G_{37}$},
\end{cases}
\end{equation*}

\begin{equation*}
y_2^{2^s}+
\begin{cases}
a^2+ac &\text{for $G=G_{36}$},\\
a^2+ac+c^2 &\text{\,\,\,$G=G_{37}$},
\end{cases}
\end{equation*}

where
\begin{equation*}
x_1=(x_2+x_1x_2^{2^{s-1}})^{2^{s-1}}+
\begin{cases}
b & \text{for $G=G_{36}$}\\
b+c+(bc)^{2^{s-1}} & \text{\,\,\,$G=G_{37},$}
\end{cases}
\end{equation*}

\begin{equation*}
y_1=(y_2+y_1y_2^{2^{s-1}})^{2^{s-1}}+
\begin{cases}
c& \text{for $G=G_{36}$} \\
0& \text{\,\,\, $G=G_{37}$}.
\end{cases}
\end{equation*}

ii) Some other relations are

$a^2c=ac^2,\,\,\,$ $b^2c=bc^2,\,\,\,$ $x_1^{2^{s}}=b^{2^{s-1}}c^{2^{s-1}},\,\,\,$ $y_1^{2^{s}}=a^{2^{s-1}}c^{2^{s-1}}.$
\end{theorem}

\bigskip

%
%
%
%
%
%
%
%
%
%
%

\bigskip

\section{Group $G_{34}$ and its non-split version $G_{35}$}

Presentations of $G_{34}$ and $G_{35}$ are as follows:
\begin{multline*}
\shoveleft{}\\
\shoveleft{G_{34}=\langle \mathbf{a},\mathbf{b},\mathbf{c} \mid  \mathbf{a}^4=\mathbf{b}^4=\mathbf{c}^2=[\mathbf{a},\mathbf{b}]=1, \mathbf{c}\mathbf{a}\mathbf{c}=\mathbf{a}^{-1}, \mathbf{c}\mathbf{b}\mathbf{c}=\mathbf{b}^{-1} \rangle ,}\\
\shoveleft{G_{35}=\langle \mathbf{a},\mathbf{b},\mathbf{c} \mid  \mathbf{a}^4=\mathbf{b}^4=[\mathbf{a},\mathbf{b}]=1,\mathbf{c}^2=\mathbf{a}^2, \mathbf{c}\mathbf{a}\mathbf{c}^{-1}=\mathbf{a}^{-1},\mathbf{c}\mathbf{b}\mathbf{c}^{-1}=\mathbf{b}^{-1} \rangle .}\\
\end{multline*}
Let $G$ be either $G_{34}$ or $G_{35}$ and let $H=\langle \mathbf{a},\mathbf{b}\rangle\cong C_4 \times C_4 $ be the maximal abelian subgroup in $G$.
Let $\lambda, \nu\to BH$ denote the pullbacks of the canonical complex line bundles along the projections onto the first and second  factor of $H$ respectively:
$$\lambda(\mathbf{a})=\nu(\mathbf{b})=i,\,\,\, \lambda(\mathbf{b})=\lambda(c)=\nu(\mathbf{a})=\nu(\mathbf{c})=1.$$

The quotient of $G$ by the center $Z\cong C_2^2$ is isomorphic to $C_2^3$. The projections of $G$ on the three factors induce three line bundles 
$\alpha$, $\beta$ and $\gamma$ over $BG$: 

$$\alpha(a)=\beta(b)=\gamma(c)=-1,\,\,\,\alpha(b)=\alpha(c)=\beta(a)=\beta(c)=\gamma(a)=\gamma(b)=1.$$

For both cases let us denote Chern classes by

\begin{equation*}
x_i=c_i(Ind_H^G(\lambda)); \,\,\,y_i=c_i(Ind_H^G(\nu));
\end{equation*}

\begin{equation*}
a=c_1(\alpha),\,\,\,b=c_1(\beta),\,\,\,c=c_1(\gamma).
\end{equation*}

\medskip

Let $Tr^*: K(s)^*(BH)\to K(s)^*(BG)$ be the transfer homomorphism \cite{A} associated to the double covering $\rho: BH \to BG$ and let
$$
T=Tr^*(uv),
\text{ where } u=c_1(\lambda);\,\,\, v=c_1(\nu).
$$

\medskip

\begin{theorem}
\label{thm:G34-35}  Let $G$ be one of the groups $G_{34},G_{35}$. Then

i) $K(s)^*(BG)\cong K(s)^*[a,b,c,x_2,y_2,T]/R$, where the ideal $R$ is generated by

$a^{2^s}$, $b^{2^s}$, $c^{2^s}$,

$c(c+x_1+\sum_{i=1}^{s-1}c^{2^s-2^i}x_2^{2^{i-1}})$,
\,\,$c(c+y_1+\sum_{i=1}^{s-1}c^{2^s-2^i}y_2^{2^{i-1}})$,

$a(a+y_1+\sum_{i=1}^{s-1}a^{2^s-2^i}y_2^{2^{i-1}})$,\,\,\,
$b(b+x_1+\sum_{i=1}^{s-1}b^{2^s-2^i}x_2^{2^{i-1}})$,

$(c+x_1+\sum_{i=1}^{s-1}c^{2^s-2^i}x_2^{2^{i-1}})(b+y_1+\sum_{i=1}^{s-1}b^{2^s-2^i}y_2^{2^{i-1}})+b^{2^s-1}T$,

$(c+y_1+\sum_{i=1}^{s-1}c^{2^s-2^i}y_2^{2^{i-1}})(a+x_1+\sum_{i=1}^{s-1}a^{2^s-2^i}x_2^{2^{i-1}})+a^{2^s-1}T$,

$T^2+Tx_1y_1+x_2y_1(c+y_1+\sum_{i=1}^{s-1}c^{2^s-2^i}y_2^{2^{i-1}})+x_1y_2(c+x_1+\sum_{i=1}^{s-1}c^{2^s-2^i}x_2^{2^{i-1}})$,

$T(a+x_1+\sum_{i=1}^{s-1}a^{2^s-2^i}x_2^{2^{i-1}})+a^{2^s-1}x_2(c+y_1)$,

$T(b+y_1+\sum_{i=1}^{s-1}b^{2^s-2^i}y_2^{2^{i-1}})+b^{2^s-1}y_2(c+x_1)$,\,\,\,$cT$,

$x_2^{2^s}+b^2+bc$,\,\,\, $y_2^{2^s}+a^2+ac,$

\medskip

where
\begin{equation*}
x_1=(x_2+x_1x_2^{2^{s-1}})^{2^{s-1}}+
\begin{cases}
c & \text{for $G=G_{34}$}\\
0 & \text{\,\,\,$G=G_{35}$}
\end{cases}
\end{equation*}

and
\begin{equation*}
y_1=c+(y_2+y_1y_2^{2^{s-1}})^{2^{s-1}},\text{ for both cases.}
\end{equation*}

\medskip

ii) Some other relations are

$a^2c=ac^2,\,\,\,$ $b^2c=bc^2,\,\,\,$ $x_1^{2^{s}}=b^{2^{s-1}}c^{2^{s-1}},\,\,\,$ $y_1^{2^{s}}=a^{2^{s-1}}c^{2^{s-1}}.$
\end{theorem}

\bigskip

\bigskip

\bigskip

\section{Proof of Theorem \ref{thm:G} }
\subsection*{Bundle relations}

Let us give some relations of bundles over $BG$ we will need. We omit the proofs since they are completely standard and easily follow from the definitions and from Frobenius reciprocity of the transfer in complex $K$-theory.

Let $\rho:BH \rightarrow BG$ be the double covering $\rho=\rho(H,G)$ and let $\rho_!\lambda=Ind^G_H(\lambda)$ and $\rho_!\nu=Ind^G_H(\nu)$ in both cases. Both groups have the same character table. The only difference is in the determinants of $\rho_!\lambda$ and $\rho_!\nu$, while restrictions, products, and two lemmas below are the same. Namely

\medskip
\begin{align*}
&det(\rho_!\lambda)=\gamma, \,\,\, det(\rho_!\nu)=\beta, &&\text{ for }G=G_{36},\\
&det(\rho_!\lambda)=1, \,\,\, det(\rho_!\nu)=\beta \gamma, && \text{ for } G=G_{37},
\end{align*}

\medskip

and for both cases is easily checked the following
\begin{align*}
&\text{Restrictions}                                                                     && \text{Product relations}  \\
i)\,\,\,\, \rho^*\alpha &=\lambda^ {2}, \,\,\, \rho^*\beta=\mu, \,\,\, \rho^*\gamma=1; &&  iii) \,\,\,\, \alpha \rho_!\lambda=\rho_!\lambda,\,\,\,\gamma \rho_!\lambda=\rho_!\lambda ;\\
ii)\,\,\,\, \rho^*(\rho_!\lambda) &=\lambda +\lambda^3,\,\,\, \rho^*(\rho_!\nu)=\nu + \nu \mu; &&  iv)  \,\,\,\,  \beta \rho_!\nu=\rho_!\nu,\,\,\,\gamma \rho_!\nu=\rho_!\nu; \\
&                                                                                 &&  v)\,\,\,\,(\rho_!\nu)^2=1 + \beta + \gamma + \beta \gamma ;\\                                                  &                                                                                && vi)\,\,\,\,(\rho_!\lambda)^2=1+ \alpha + \gamma + \alpha \gamma.
\end{align*}

The first relation of iii) suggests that $\rho_!\lambda$ should be also a transfer line bundle over the total space of a  $2$-covering corresponding to $\alpha$. Namely we have

\medskip

\begin{lemma}
\label{lem:LAMBDA-36}
Let $G=G_{36}$. There exists a subgroup $H'\subset G_{36}$ and a line bundle $\lambda' \to BH'$ such that $Ind_H^{G}(\lambda)=Ind_{H'}^{G}(\lambda').$ The bundle $\alpha$ restricts to trivial bundle over $BH'$.
\end{lemma}

\begin{proof}
One can choose the group $H'=\langle \mathbf{b}^2,\mathbf{c},\mathbf{a}\rangle $ and the line bundle $\lambda'$ represented by $\lambda'(\mathbf{a})=1$, $\lambda'(\mathbf{b}^2)=-1$, $\lambda'(\mathbf{c})=1$.
\end{proof}

We will need also the following three easy facts.

\begin{lemma}
\label{lem:NU-36}
Let $G=G_{36}$. There exists $K'\subset G$ and a line bundle $\nu' \to BK'$ such that $Ind_H^{G}(\nu)=Ind_{K'}^{G}(\nu').$ The bundle $\beta$ restricts to trivial bundle over $BK'$.
\end{lemma}
%
\begin{proof} Put $K'=\langle b,a \rangle$ and $\nu'(b)=1$, $\nu'(a)=i$.
\end{proof}
%

\medskip

\begin{lemma}
\label{lem:LAMBDA-37}
Let $G=G_{37}$. There exists $H^{''}\subset G_{37}$ and a line bundle $\lambda^{''} \to BH^{''}$ such that $Ind_H^{G}(\lambda)=Ind_{H^{''}}^{G}(\lambda^{''}).$ The bundle $\alpha\gamma$ restricts to trivial bundle over $BH^{''}$.
\end{lemma}

\begin{proof} Let $H^{''}=\langle \mathbf{b}^2,\mathbf{c},\mathbf{ab}\rangle $ and $\lambda''$ be represented by $\lambda^{''}(\mathbf{ab})=i$, $\lambda^{''}(\mathbf{b^2})=-1$,  $\lambda^{''}(\mathbf{c})=1$.
\end{proof}

\medskip

\begin{lemma}
\label{lem:NU-37}
Let $G=G_{37}$. There exists $K^{''}\subset G$ and a line bundle $\nu^{''} \to BK^{''}$ such that $Ind_H^{G}(\nu)=Ind_{K^{''}}^{G}(\nu^{''}).$ The bundle $\beta$ restricts to trivial bundle over $BK^{''}$.
\end{lemma}

\begin{proof}
Let $K^{''}=\langle \mathbf{b},\mathbf{d},\mathbf{a}\rangle $ and $\nu^{''}(\mathbf{a})=\nu^{''}(\mathbf{b})=1$,  $\nu^{''}(\mathbf{d})=-1$.
\end{proof}

\bigskip

\bigskip

\subsection{Relations of Theorem \ref{thm:G}}

Clearly the relations
$$
a^{2^s}=b^{2^s}=c^{2^s}=0.
$$
are immediate consequences of the bundle relations $\alpha^2=\beta^2=\gamma^2=1$ for all cases.

The 4th and 5th relations follow from \eqref{Tru} and \eqref{Trv} respectively.

For the 6th relation
\begin{equation*}
a(a+y_1+\sum_{i=1}^{s-1}a^{2^s-2^i}y_2^{2^{i-1}})=0
\end{equation*}
consider the transfer $Tr^{'*}$ of the double covering $\rho:BH' \to BG$ in case $G_{36}$ and apply formula \eqref{eq:tr} and Lemma \ref{lem:LAMBDA-36}.  Then it follows that the second factor of the relation is the transfer of $c_1(\lambda')$, i.e.,
$$Tr^{'*}(c_1(\lambda'))=(a+y_1+\sum_{i=1}^{s-1}a^{2^s-2^i}y_2^{2^{i-1}}).$$
Therefore
$$aTr^{'*}(c_1(\lambda'))=Tr^{'*}(\rho^*(a)c_1(\lambda'))=Tr^{'*}(0\cdot c_1(\lambda'))=0.$$
Similarly, for $G_{37}$ apply formula \eqref{eq:tr} and Lemma \ref{lem:LAMBDA-37}.

\medskip

In the same spirit, apply formula \eqref{eq:tr}, Lemma \ref{lem:NU-36} and Lemma \ref{lem:NU-37}
for the 7th relation
\begin{equation*}
b(b+x_1+\sum_{i=1}^{s-1}b^{2^s-2^i}x_2^{2^{i-1}})=0.
\end{equation*}

\medskip

Now note that the 4th and 6th relations imply $a^2c=ac^2$, the first relations of Theorem \ref{thm:G} ii). For this multiply the 4th relation by $a$ and the 6th relation by $c$. The sum of these terms equals $a^2c+ac^2$ up to an invertible factor.  Similarly the 5th and 7th relation imply $b^2c=bc^2$, the second relation of Theorem \ref{thm:G} ii). Note also

\begin{equation}
\label{a^ic^j}
a^ic^j=0, \,\,\, b^ic^j=0,\,\, i+j>2^s.
\end{equation}

The first one is the consequence of $a^2c=ac^2$ and $a^{2^s}=c^{2^s}=0$. Similarly the second one follows from $b^2c=bc^2$ and $b^{2^s}=c^{2^s}=0$.

\medskip

For the decompositions of $x_2^{2^s}$, $y_2^{2^s}$, (also for the formulas for $x_1^{2^s}$ and $y_1^{2^s}$ of Theorem \ref{thm:G} ii) we need the material of Section 3. Namely we have to apply Lemma \ref{lem:zeta^2} to all induced representations given in Section 3 and take into account that their determinants can written in terms of the bundles $\alpha, \beta, \gamma$.
For example for $G=G_{36}$

\begin{equation*}
x_2^{2^{s}}=c^2+bc, \,\,\, x_1^{2^s}=b^{2^{s-1}}c^{2^{s-1}}
\end{equation*}
and
\begin{equation*}
y_2^{2^s}=a^2+ac, \,\,\, y_1^{2^s}=a^{2^{s-1}}c^{2^{s-1}}
\end{equation*}
are the consequences of product relations v) and vi).
Let us prove first two relations. Equate Chern classes in the bundle relation of v). Then taking into account \eqref{a^ic^j} we get
from the equation for the first Chern classes
$$
x_1^{2^s}=b+c+b+c+b^{2^{s-1}}c^{2^{s-1}}=a^{2^{s-1}}c^{2^{s-1}}.
$$

For the decomposition of $x_2^{2^s}$ apply the equation for the second Chern classes and take into account \eqref{a^ic^j}:

\begin{align*}
x_2^{2^s}&=c_2(\rho_!\nu^2)+c_1(det\,\rho_!\nu)^2 && \\
&=c_2(1+\beta+\gamma+\beta \gamma)+c_1(\beta)^2=bc+bF(b,c)+cF(b,c)+b^2 && \\
&=bc+c^2. && \\
\end{align*}

\medskip
%
%
%
%
%
%
%
%
%
%
%

\medskip

The 8th and 9th relations.

\begin{equation*}
\label{eq:Tb3}
(c+y_1+\sum_{i=1}^{s-1}c^{2^s-2^i}y_2^{2^{i-1}})(b+x_1+\sum_{i=1}^{s-1}b^{2^s-2^i}x_2^{2^{i-1}})=Tb^{2^s-1}.
\end{equation*}

\begin{proof}
Let $G=G_{36}$. Consider the diagram

\medskip

\begin{equation}
\begin{CD}
\label{eq:DG36b}
B \langle \mathbf{b},\mathbf{a}^2 \rangle @>>> B \langle \mathbf{b},\mathbf{a} \rangle \\
@VV{\rho_{\mu}}V              @VV{\rho_{\beta}}V \\
B \langle \mathbf{b},\mathbf{c},\mathbf{a}^2 \rangle @>>\rho_{\gamma}> BG
\end{CD}
\end{equation}

\medskip

Then the left hand side of our relation is equal to
\begin{align*}
&Tr^*_{\gamma}(v)(b+x_1+\sum_{i=1}^{s-1}b^{2^s-2^i}x_2^{2^{i-1}})               && \text{by transfer formula \eqref{eq:tr}} \\
&=Tr^*_{\gamma}(v) Tr^*_{\beta}(c_1(\nu'))                                          && \text{by Lemma \ref{lem:NU-36} } \\
&=Tr^*_{\gamma}(v \cdot \rho^*_{\gamma}Tr^*_{\beta}(c_1(\nu')))                      && \text{by Frobenius reciprocity of the transfer} \\
&=Tr^*_{\gamma}(v \cdot Tr^*_{\mu}(\rho^*_{\mu}(u)))                         && \text{by the double coset formula and Lemma \ref{lem:NU-36}} \\
&=Tr^*_{\gamma}(Tr^*_{\mu}(\rho^*_{\mu}(uv)))                                                 && \text{by Frobenius reciprocity} \\
&=Tr^*_{\gamma}(uv \cdot Tr^*_{\mu}(1))=Tr^*_{\gamma}(uv \cdot c_1^{2^{s}-1}(\mu))         && \text{by the formula for $Tr^*(1)$} \\
&=Tb^{2^{s}-1}.                                                                       && \text{by the definitions of $\beta$, $\mu$, and $T$} \\
\end{align*}
\end{proof}

Similarly

\begin{equation*}
\label{eq:Ta3}
(c+x_1+\sum_{i=1}^{s-1}c^{2^s-2^i}x_2^{2^{i-1}})(a+y_1+\sum_{i=1}^{s-1}a^{2^s-2^i}y_2^{2^{i-1}})=Ta^{2^s-1}.
\end{equation*}

\begin{proof} Consider the diagram

\medskip

\begin{equation}
\label{eq:DG36a}
\begin{CD}
B\langle \mathbf{b}^2,\mathbf{c},\mathbf{a}^2\rangle           @>>>         B\langle \mathbf{a},\mathbf{b}^2,\mathbf{c}\rangle\\
@VV{\rho_{\lambda^2}}V                                                             @VV{\rho_{\alpha}}V \\
B\langle \mathbf{b},\mathbf{c},\mathbf{a}^2 \rangle             @>>\rho_{\gamma}>         BG
\end{CD}
\end{equation}

\medskip

With this notation the left hand side of the above relation is equal to
\begin{align*}
&Tr^*_{\gamma}(u)  Tr^*_{\alpha}(c_1(\lambda'))                                && \text{by Lemma \ref{lem:LAMBDA-36} and formula \eqref{eq:tr}}\\
&=Tr^*_{\gamma}(u \cdot \rho^*_{\gamma}Tr^*_{\alpha}(c_1(\lambda')))                    && \text{by Frobenius reciprocity of the transfer}\\
&=Tr^*_{\gamma}(u \cdot Tr^*_{\lambda^2}(\rho^*_{\lambda^2}(v)))              && \text{by the double coset formula and Lemma \ref{lem:LAMBDA-36}}\\
&=Tr^*_{\gamma}(Tr^*_{\lambda^2}(\rho^*_{\lambda^2}(uv)))                                  && \text{by Frobenius reciprocity}\\
&= Tr^*_{\gamma}(uv \cdot Tr^*_{\lambda^2}(1))=Tr^*_{\gamma}(uv \cdot c_1^{2^{s}-1}(\lambda^2))  && \text{by the formula for $Tr^*(1)$}\\
&=Ta^{2^{s}-1}                                                                          && \text{by the definitions of $\alpha$, $\lambda$, and $T$.}\\
\end{align*}
\end{proof}

For the case of $G=G_{37}$ the proof is analogous. We just have to apply Lemma \ref{lem:NU-37} and Lemma \ref{lem:LAMBDA-37}.

\medskip

The same applies to 11th and 12th relations and may be proved for both groups simultaneously.
In each case we will arrive at

$$T(a+y_1+\sum_{i=1}^{s-1}a^{2^s-2^i}y_2^{2^{i-1}})=a^{2^s-1}Tr^*(uv^2)$$
or
$$T(b+x_1+\sum_{i=1}^{s-1}b^{2^s-2^i}x_2^{2^{i-1}})=b^{2^s-1}Tr^*(u^2v).$$
Therefore we will need that for the involution $t \in C_2=G/H$ one has by Frobenius reciprocity

\medskip

\noindent i) $ Tr^*(u^2v)=Tr^*(uv(u+tu)-vutu))=Tr^*(uv)x_1-Tr^*(v)x_2$,

\noindent ii) $ Tr^*(uv^2)=Tr^*(uv(v+tv)-uvtv))=Tr^*(uv)y_1-Tr^*(u)y_2.$

\bigskip

Here is details for $G=G_{37}$. Apply the diagram


\begin{equation}
\begin{CD}
\label{eq:DG37b}
B \langle \mathbf{b},\mathbf{d} \rangle @>>> B \langle \mathbf{b},\mathbf{d},\mathbf{a} \rangle \\
@VV{\rho_{\mu}}V              @VV{\rho_{\beta}}V \\
B \langle \mathbf{b},\mathbf{c},\mathbf{d}\rangle @>>\rho_{\gamma}> BG
\end{CD}
\end{equation}

Then the above equality i) gives
$$
T(b+x_1+\sum_{i=1}^{s-1}b^{2^s-2^i}x_2^{2^{i-1}})+b^{2^s-1}Tx_1+b^{2^s-1}x_2(c+y_1+\sum_{i=1}^{s-1}c^{2^s-2^i}y_2^{2^{i-1}})=0.
$$
Then the second summand is zero by the 6th relation and \eqref{a^ic^j}. The third summand is equal to $b^{2^s-1}x_2(c+y_1)$ by \eqref{a^ic^j}. This gives the 11th relation.

\medskip

Similarly, applying the diagram

\begin{equation}
\label{eq:DG37a}
\begin{CD}
B\langle \mathbf{b},\mathbf{c}\rangle           @>>>         B\langle \mathbf{ab},\mathbf{a}^2,\mathbf{c}\rangle\\
@VV{\rho_{\lambda^2}}V                                                             @VV{\rho_{\alpha\gamma}}V \\
B\langle \mathbf{b},\mathbf{c},\mathbf{d} \rangle             @>>\rho_{\gamma}>         BG
\end{CD}
\end{equation}

%
%
%
and the above equality ii) we have

$$
T(a+y_1+\sum_{i=1}^{s-1}a^{2^s-2^i}y_2^{2^{i-1}})+a^{2^s-1}Ty_1+a^{2^s-1}y_2(c+x_1+\sum_{i=1}^{s-1}c^{2^s-2^i}x_2^{2^{i-1}})=0.
$$

Again the second summand is zero by the 7th relation and \eqref{a^ic^j} and the third summand is equal to $a^{2^s-1}y_2(c+x_1)$. The 12th relation follows.

\qed

\medskip

The relation $cT=0$ is easy:

$$cT\equiv cTr^*_{\gamma}(uv)=Tr^*_{\gamma}(uv \gamma^*(c))=Tr^*_{\gamma}(uv \cdot 0)=0.$$

\medskip

Now let us prove the 10th relation.

Let $u'=tu$ and $v'=tv$, where $t$ is the involution as in \eqref{t} and $Tr^*=Tr^*_{\gamma}$. 

Then

\begin{align*}
&Tr^*(uv)+Tr^*(uv')=Tr^*(u(v+v'))=Tr^*(u)Tr^*(v),&&
\end{align*}

\begin{align*}
Tr^*(uv)Tr^*(uv')&=Tr^*(uv(uv'+u'v))  &&\\
&=Tr^*(u^2 vv')+Tr^*(v^2 uu')=Tr^*(u^2)y_2+Tr^*(v^2)x_2.&&
\end{align*}

Also

\medskip

\begin{align*}
Tr^*(u^2)=&Tr^*(u(u+u')-uu')=Tr^*(u)x_1-Tr^*(1)x_2,\\
Tr^*(v^2)=&Tr^*(v(v+v')-vv')=Tr^*(v)y_1-Tr^*(1)y_2.\\
\end{align*}

\medskip

Now we apply these formulas and take into account $Tr^*(1)x_1=Tr^*(1)y_1=0$. This gives the quadratic equation in $T=Tr^*_{\gamma}(uv)$

\begin{align*}
T^2&=T(c+x_1+\sum_{i=1}^{s-1}c^{2^s-2^i}x_2^{2^{i-1}})(c+y_1+\sum_{i=1}^{s-1}c^{2^s-2^i}y_2^{2^{i-1}})&&\\
&+x_2y_1(c+y_1+\sum_{i=1}^{s-1}c^{2^s-2^i}y_2^{2^{i-1}})+x_1y_2(c+x_1+\sum_{i=1}^{s-1}c^{2^s-2^i}x_2^{2^{i-1}}).&&
\end{align*}
\medskip

Now to get the 10th relation apply $cT=0$.
\qed

\medskip

\medskip

The decompositions for $x_1$ and $y_1$ are the consequences of the formula \eqref{eq:FGL} applied to the determinants of $\rho_!\nu$ and $\rho_!\lambda$.

We need $(x_1x_2)^{2^{2s-2}}=0$ and $(y_1y_2)^{2^{2s-2}}=0$. It follows from the relations ii) of Theorem \ref{thm:G} that moreover we have $(x_1x_2)^{2^s}=(y_1y_2)^{2^s}=0$: decompositions $x_1^{2^s}=(bc)^{2^{s-1}}$ and $y_1^{2^s}=(ac)^{2^{s-1}}$ imply
$$
x_1^{2^s}b=x_1^{2^s}c=x_1^{2^s}a^2=y_1^{2^s}a=y_1^{2^s}c=y_1^{2^s}b^2=0,
$$
as $a^2c=ac^2$, $b^2c=bc^2$ and $a^{2^s}=b^{2^s}=c^{2^s}=0.$

\medskip

That is, all the terms of the above decomposition of $x_2^{2^s}$ annihilate $x_1^{2^s}$. Similarly for $y_1$ and $y_2$.
It is also clear that for computing the Euler classes of the determinants (see Section 3) in each case we need only the initial fragment of the formal group law $F(x,y)=x+y+(xy)^{2^{s-1}}.$

For $G_{36}$ we have to apply the formula \eqref{eq:FGL} and take into account that $det(\rho_!\lambda)=\gamma$ and $det(\rho_!\nu)=\beta$. 

For $G_{37}$ we have $det(\rho_!\lambda)=1$ and $det(\rho_!\nu)=\beta\gamma$,  hence $e(det(\rho_!\lambda))=1$ and  $e(det(\rho_!\lambda))=F(b,F(b,c))=b+c+(bc)^{2^{s-1}}$.

\qed

\medskip

\subsection{Generators}

Here we prove the following

\begin{lemma}
\label{generatorsForG34-37}
Let $G$ be one of the groups $G_{34},\dots, G_{37}$. Then $K(s)^*(BG)$ is generated by $c,a,b,x_2,y_2,T$ as a $K(s)^*(pt)$ algebra.
\end{lemma}

\begin{proof}

Recall $G_{34}$ and $G_{35}$ have maximal abelian subgroup $H=\langle\mathbf{a},\mathbf{b}\rangle \cong C_4\times C_4$ on which the quotient acts diagonally by inverting $\mathbf{a}$ and $\mathbf{b}$. The dihedral group $D_{8}$ could be written as a semidirect product $C_4\rtimes C_2$ with that action.
Thus as a $\mathcal{C}_2$ module $K(s)^*(BH)=M\otimes M$, where $M=K(s)^*(BC_4)$.

Consider the decomposition of  $K(s)^*(BH)$ into free and trivial $C_2$-modules
$$[K(s)^*(BH)]^{C_2}=(F)^{C_2}+T.$$

Denote the restrictions of the generators of Theorem \ref{thm:G} to $K(s)^*(BH)$ with the same symbols but with bars.

Let $A$ be the subalgebra in $K(s)^*(BG)$ generated by $c,a,b,x_2,y_2,T$. Clearly the composition $\rho^*Tr^*=\text{Norm}=1+\text{involution}$ is onto
$(F)^{C_2}$. But $Im Tr^*\subset A$: As $u^2=u\bar{x_1}-\bar{x_2}$ and $v^2=v\bar{y_1}-\bar{y_2}$ any polynomial in $u,v$ can be uniquely written as  $g_0+g_1u+g_2v+g_3uv$ where $g_i=g_i(\bar{x_1},\bar{y_1},\bar{x_2},\bar{y_2}).$ Therefore one can use Frobenius reciprocity of the transfer and write transfer of any element in $Tr^*(u)$,$Tr^*(v)$,$Tr^*(uv)$, and $c,x_1,x_2,y_1,y_2$. Then one can use the formulas \eqref{Tru}, \eqref{Trv} for $Tr^*(u)$, $Tr^*(v)$ respectively and the decompositions of $x_1$ and $y_1$ of Theorem \ref{thm:G} and see that 

\begin{equation}
\label{ImTr}
Im Tr^* \subset A.
\end{equation} 
Therefore $\rho^*(A)$ is onto $(F)^{C_2}$.

Then we have to check whether the invariants in $T$ are also covered by $\rho^*A$. Let $\chi_s(F^{C_2})$ and $\chi_s(T)$ be the $K(s)^*$-Euler characteristics. Then
$$\chi_s(BH)=2\chi_s(F)+\chi_s(T), \,\,\,\chi_s(BH)^{C_2}=\chi_s(F)+\chi_s(T).$$

Let $M^{C_2}=(F')^{C_2}+T'$.
%
%
Then one can easily read off the bases for $(F')^{C_2}$ and $T'$. In notations of \cite{BV}, p. 3712, the base sets for $T'$ and $F'$ are
as follows

\begin{align*}
\chi_s(F'^{C_2})&=2^{2s-1}-2^{s-1}, &&\{c_1c_2^j\, | \,0 \leq j < 2^{s-1},
c_2^k, k\neq j, k\neq j+2^{s-1}\},\\
\chi_s(T')&=2^{s}, &&\{c_2^i,xc_2^i\,\,|\, 0\leq i\leq 2^{s-1}-1\}.\\
\end{align*}

The latter fact with \eqref{ImTr} already completes the proof of Lemma \ref{generatorsForG34-37} if apply Theorem \ref{1}. Moreover, actually we also computed  $\chi_{s}(G)$: One has
$$\chi_s(T)=\chi_s(T')^2=4^{s} \text{ \,\,\,and\,\,\, } \chi_s(F)=2\chi_s(F')\chi_s(T')+2\chi_s(F')^2=
\frac{1}{2}{16^s}-\frac{1}{2}4^{s}.$$
Thus we get already known fact from \cite{SCH1}
\begin{equation}
\label{xi}
\chi_{s}(G)=\chi_s(F)+2^s\chi_s(T)=\frac{1}{2}16^{s}-\frac{1}{2}4^{s}+8^{s}.
\end{equation}

\medskip

$G_{36}$ contains the maximal abelian subgroup
$H=\langle \mathbf{b},\mathbf{a}^2,\mathbf{c}\rangle \cong C_4 \times C_2\times C_2$. From the definition one reads off that $K(s)^*(BH)=M\otimes N$, where $N=K(s)^*(BC_2\times BC_2)$, with the switch action. The situation is similar for $G_{37}$ and $H=\langle \mathbf{b}, \mathbf{c}, \mathbf{d}\rangle$.

\end{proof}

\subsection{End of the proof}

 After Lemma \ref{generatorsForG34-37} it suffices to verify that for all our groups the defining relations of Theorem \ref{thm:G} give us a ring of Euler characteristic \eqref{xi}, already computed in \cite{SCH1}.
For each of our groups, we have to choose a basis for $K(s)^*(BG)$ as the union of a basis for $K(s)^*(BG)/ker \rho^*$, and a basis for $ker \rho^*,$ where $\rho:BH\rightarrow BG$.

\begin{lemma}
\label{basis36-37}

A basis for $K(s)^*(BG)$, $G=G_{36},G_{37}$ is

 $\{x_1^iy_1^jx_2^ky_2^l | i,j<2^s, k,l<2^{s-1}\}$;

 $\{ax_2^{k+2^{s-1}}y_2^l | k,l<2^{s-1}\}$;

 $\{x_1^iax_2^ky_2^l$, $y_1^ix_2^{k+2^{s-1}}y_2^l | i<2^s, k,l<2^{s-1}\}$;

 $\{Tx_1^iy_1^jx_2^ky_2^l |i,j<2^s-1, k,l<2^{s-1}\}$;

 $\{c^ix_2^jy_2^k,\ c^iax_2^jy_2^k | 0<i<2^s, j<2^s, k<2^{s-1}\}$.
\end{lemma}

\begin{proof}
One can verify that a) the first four sets give a basis for $K(s)^*(BG)/ker \rho^*$, and b) the last set forms a basis for $ker \rho^*$.
\newline
For a) choose the lexicographic ordering corresponding to $(T,a,b,y_2,x_2,y_1,x_1,c)$ in that order.
%
%
%
%
\newline
For b) choose the lexicographic ordering corresponding to $(T,x_1,y_1,a,b,x_2,y_2,c),lp$. To read off the required basis b)
we have to do one more thing. Namely we have replace the monomials
$x_2^jy_2^{2^{2s-1}+k}=x_2^jy_2^{2^{2s-1}}y^{k}$
by
$ac^{2^s-1}x_2^jy^{k}$, $j<2^s$, $k<2^{s-1}$.
This can be done by relations of Theorem \ref{thm:G}
$$y_2^{2^{2s-1}}=y_1^{2^{s}}=(ac)^{2^{s-1}}=ac^{2^{s}-1}.$$

Note the last set of Lemma \ref{basis36-37} is the union of two sets
\newline
$\{c^ix_2^jy_2^k,\, c^iax_2^jy_2^k, | 0<i<2^s-1\}$, a basis of $ker \rho^* \bigcap K(s)^*(BG)/Im Tr^*$, and
\newline
$\{c^{2^s-1}x_2^jy_2^k,\,\,c^{2^s-1}ax_2^jy_2^k\} $, a basis of $Tr^*(T)$, the image of the trivial module $T$. For the last sentence recall $Tr^*(1)=v_sc^{2^s-1}.$
%
%
\end{proof}

\section{Proof of Theorem \ref{thm:G34-35}}
The proof could be rather short, since it uses the same arguments.
Let us give some relations of bundles we will need. Even if no explicit use is made of these relations in
the text, it does help to have them on record should the reader be inclined
to spell out proofs for $G=G_{34},G_{35}$.

\subsection*{Bundle relations}

Let $\rho:BH \rightarrow BG$ be the double covering $\rho=\rho(H,G)$ and let $\rho_{!}\lambda=Ind^G_H(\lambda)$ and $\rho_{!}\nu=Ind^G_H(\nu)$ in both cases.

\medskip

One has for $G=G_{34}$

\begin{align*}
&\text{Determinants}                                        && \text{Product relations}  \\
&det(\rho_!\lambda)=\gamma, \,\, det(\rho_!\nu)=\gamma;             &&  iii) \,\,\,\, \alpha \rho_!\lambda=\rho_!\lambda,\,\,\,\gamma \rho_!\lambda=\rho_!\lambda; \\
&\text{Restrictions}                                        &&  iv)  \,\,\,\,  \beta \rho_!\nu=\rho_!\nu,\,\,\,\gamma \rho_!\nu=\rho_!\nu; \\
&i)\,\,\,\, \rho^*\alpha =\lambda^{2}, \,\,\,
\rho^*\beta=\nu^2, \,\,\, \rho^*\gamma=1;                     &&  v)\,\,\,\,(\rho_!\nu)^2=1 + \beta + \gamma + \beta \gamma ;\\
&ii)\,\,\,\, \rho^*(\rho_!\lambda) =\lambda +\lambda^3,\,\,\,
\rho^*(\rho_!\nu)=\nu + \nu^3;                                      && vi)\,\,\,\,(\rho_!\lambda)^2=1+ \alpha + \gamma + \alpha \gamma.                                                                                                            \end{align*}

\medskip

$G_{35}$ has the same character table as $G_{34}$. The only difference is in the determinant of $\lambda_{!}$, while restrictions, products, and the two lemmas above are the same.
$$
det(\rho_!\lambda)=1, \,\, det(\rho_!\nu)=\gamma.
$$

The first relation of iii) suggests that $\rho_!\lambda$ should be also transfer of a line bundle over the total space of a $2$-covering corresponding to $\alpha$. Namely, we have

\medskip

\begin{lemma}
\label{lem:lAMBDA-34-35}
 Let $G=G_{34}$ and $G=G_{35}$ . There exists $H'\subset G$ and a line bundle $\lambda' \to BH'$ such that $Ind_H^{G}(\lambda)=Ind_{H'}^{G}(\lambda').$ The bundle $\alpha$ restricts to trivial bundle over $BH'$.
\medskip
\end{lemma}

\label{lem:lAMBDA-34-35}
\begin{proof}
One can choose the group $H'=\langle \mathbf{a}^2,\mathbf{b},\mathbf{c}\rangle $ and the line bundle $\lambda'$ represented by $\lambda'(\mathbf{a}^2)=-1$, $\lambda'(\mathbf{b})=1$, $\lambda'(\mathbf{c})=1$ ;  $\mu(\mathbf{a}^2)=1$, $\mu(\mathbf{b})=1$, $\mu(\mathbf{c})=-1$.
\end{proof}

Similarly we have the following lemma for $\rho_!\nu$  by the first equality of iv).

\begin{lemma}
\label{lem:NU-34-35}
 Let $G=G_{34}$ and $G=G_{35}$  There exists $H''\subset G$ and a line bundle $\nu' \to BH''$ such that $Ind_H^{G}(\nu)=Ind_{H''}^{G}(\nu').$ The bundle $\beta$ restricts to trivial bundle over $BH''$.
\medskip
\end{lemma}

\begin{proof}
One can choose the group $H''=\langle \mathbf{a},\mathbf{b}^2,\mathbf{c} \rangle $ and the line bundle $\nu'$ represented by $\nu'(\mathbf{a})=1$, $\nu'(\mathbf{b}^2)=-1$, $\nu'(\mathbf{c})=1$.
\end{proof}

\subsection*{End of the proof}

Using the above information on complex representations of the groups $G_{34}$ and $G_{35}$, the proof is completely similar to that of Theorem \ref{thm:G} and is left to the reader.

\bigskip

\section*{Acknowledgements}
The authors are very grateful to the referee for exceptionally thorough analysis of the paper and numerous important suggestions which have been very useful for improving the paper.

\bigskip

\bibliographystyle{amsplain}

\end{document}